\documentclass[11pt]{article}

\usepackage{graphicx,algorithmic,algorithm}
\usepackage{enumerate}
\usepackage{tikz}
\usepackage[nomessages]{fp}
\usetikzlibrary{shapes}
\usepackage{hyperref}
\hypersetup{
    pdftitle=   {Rainbow coloring of bounded expansion classes},
   pdfauthor=  {Minki Kim and Jinha Kim and O-joung Kwon}
}
\usepackage{todonotes}
\usepackage{authblk}
\usepackage{amsthm,amssymb,amsmath}
\usepackage{mathabx}
\newtheorem{theorem}{Theorem}[section]
\newtheorem{lemma}[theorem]{Lemma}
\newtheorem{proposition}[theorem]{Proposition}
\newtheorem{corollary}[theorem]{Corollary}

\newtheorem{claim}{Claim}
\newtheorem*{thm:main1}{Theorem~\ref{thm:main1}}
\newtheorem*{thm:main2}{Theorem~\ref{thm:main2}}
 \newenvironment{clproof}{\begin{list}{}{%
               \setlength{\leftmargin}{5mm}%
               } \item {\it Proof.} }{\hfill$\lozenge$\end{list}\medskip}

\newcommand\abs[1]{\lvert #1\rvert}

\newcommand{\td}{\operatorname{td}}
\newcommand{\clos}{\operatorname{clos}}
\newcommand{\dist}{\operatorname{dist}}
\newcommand{\wcol}{\operatorname{wcol}}
\newcommand{\WReach}{\operatorname{WReach}}

\newcommand{\cC}{\mathcal{C}}
\newcommand{\cF}{\mathcal{F}}

\newcommand{\cI}{\mathcal{I}}

\newcommand{\cA}{\mathcal{A}}
\newcommand{\cB}{\mathcal{B}}
\newcommand{\cH}{\mathcal{H}}
\newcommand{\cT}{\mathcal{T}}

\begin{document}
\title{Rainbow independent sets on dense graph classes}

\author{Jinha Kim\thanks{All authors are supported by Institute~for~Basic~Science (IBS-R029-C1).}}
\author{Minki Kim}

\affil[1]{\small Discrete Mathematics Group, Institute~for~Basic~Science~(IBS), Daejeon,~South~Korea.}

\author[1,2]{O-joung Kwon\thanks{O. Kwon is supported by the National Research Foundation of Korea (NRF) grant funded by the Ministry of Education (No. NRF-2018R1D1A1B07050294).}}
\affil[2]{\small Department of Mathematics, Incheon~National~University, Incheon,~South~Korea.}

\date\today
\maketitle
  \setcounter{footnote}{1}%
  \footnotetext{E-mail addresses: \texttt{jinhakim@ibs.re.kr} (J. Kim), \texttt{minkikim@ibs.re.kr} (M. Kim), \texttt{ojoungkwon@gmail.com} (Kwon) 
  }

\begin{abstract}
	Given a family $\mathcal{I}$ of independent sets in a graph, a rainbow independent set is an independent set $I$ such that there is an injection $\phi\colon I\to \mathcal{I}$ where for each $v\in I$, $v$ is contained in $\phi(v)$.
	Aharoni, Briggs, J. Kim, and M. Kim [Rainbow independent sets in certain classes of graphs. arXiv:1909.13143] determined for various graph classes $\mathcal{C}$ whether $\mathcal{C}$ satisfies a property that for every $n$, there exists $N=N(\mathcal{C},n)$ such that every family of $N$ independent sets of size $n$ in a graph in $\mathcal{C}$ contains a rainbow independent set of size $n$. 
    In this paper, we add two dense graph classes satisfying this property, namely, the class of graphs of bounded neighborhood diversity and the class of $r$-powers of graphs in a bounded expansion class. 
	\end{abstract}

\section{Introduction}

Given a family $\mathcal{F}$ of subsets of a set $A$, a subset $B$ of $A$ is an \emph{$\mathcal{F}$-rainbow set} if there is an injection $\phi\colon B\to \mathcal{F}$ such that for each $b\in B$, $b$ is contained in $\phi(b)$. 
Given a family $\mathcal{M}$ of matchings in a graph $G$, a \emph{rainbow matching} is a matching in $G$ that is an $\mathcal{M}$-rainbow set.
Rainbow matchings in bipartite graphs have attracted attention because of their connection to transversals of Latin squares: see, for example, \cite[Section $2$]{RonE2009}.
 Drisko~\cite{Drisko1998} showed that every family of $2n-1$ matchings of size $n$ in a bipartite graph, in which each side has size $n$, contains a rainbow matching of size $n$.
 Later, Aharoni and Berger~\cite{RonE2009} showed that this is true for all bipartite graphs. 
For general graphs, Aharoni, Berger, Chudnovsky, Howard and Seymour~\cite{RonEMDP2019} showed that every family of $3n-2$ matchings of size $n$ contains a rainbow matching of size $n$. They conjectured that if $n$ is even, then $2n$ matchings of size $n$ have a rainbow matching of size $n$, and if $n$ is odd, then $2n-1$ matchings have a rainbow matching of size $n$.
Recently, Aharoni, Briggs, J.~Kim and M.~Kim~\cite{AharoniBKK2021} proved that $3n-3$ matchings are sufficient for $n \geq 3$.

In this paper, we consider an analogue of rainbow matchings for independent sets. 
Given a family $\mathcal{I}$ of independent sets in a graph, a \emph{rainbow independent set} is an independent set that is an $\mathcal{I}$-rainbow set.
For the line graph $H$ of a graph $G$, rainbow matchings in $G$ correspond to rainbow independent sets in $H$.
Thus, the result of Aharoni and Berger~\cite{RonE2009} implies that every family of $2n-1$ independent sets of size $n$ in the line graph of a bipartite graph contains a rainbow independent set of size $n$.
Motivated by this fact, Aharoni, Briggs, J. Kim, and M. Kim~\cite{ABKK2019} and Kim and Lew~\cite{KimL2019} considered the same type of problems for rainbow independent sets in various graph classes. They observed that, contrary to the matching case, such a bound does not exist for all graphs. A simple example is the complete multipartite graph where each part has size $n$ and the number of parts is $k$. Clearly, this example has $k$ distinct independent sets of size $n$ which has no rainbow independent set. Thus, for the class of all complete multipartite graphs, there is no integer $N$ satisfying that every family of $N$ independent sets of size $n$ has a rainbow independent set of size $n$.

For a graph $G$ and a positive integer $n$, let $f_G(n)$ be the minimum $k$ such that every family of $k$ independent sets of size $n$ has a rainbow independent set of size $n$.
For a class $\mathcal{C}$ of graphs and a positive integer $n$, let $f_{\mathcal{C}} (n)=\sup \{ f_G(n) : G\in \mathcal{C}\}$.
We say a class $\mathcal{C}$ of graphs has the \emph{rainbow property} (for independent sets) if for every positive integer $n$, $f_{\mathcal{C}}(n)<\infty$.
As dicussed, there are graph classes that do not have the rainbow property: for example, the class of complete multipartite graphs.
It was shown in \cite{ABKK2019} that the class of $H$-induced-subgraph-free graphs has the rainbow property if and only if $H$ is either a complete graph or the graph obtained from a complete graph by removing one edge.

For well-known sparse graph classes $\mathcal{C}$, such as the class of planar graphs, we may observe that $\cC$ satisfies the rainbow property.
Ne\v{s}et\v{r}il and Ossona de Mendez~\cite{NO2008a} 
introduced the notion of \emph{bounded expansion classes of graphs}, generalizing the classes of graphs with excluding minors and the classes of graphs of bounded degree. 
The same authors introduced even more general sparse graph classes called \emph{nowhere dense classes of graphs}~\cite{NO2011}, which further include the classes with locally excluding a minor. 
We refer to \cite{NO2012} for more extensive study on sparse graph classes.
As an other extension, a bounded expansion class has bounded degeneracy, and classes of bounded degeneracy are incomparable with nowhere dense classes.
It is not difficult to see that every nowhere dense class or every class of bounded degeneracy has the rainbow property because they are $K_t$-induced-subgraph-free for some $t$ depending only on the class.

What about well-behaved dense graph classes? 
We can observe that the class of graphs that can be partitioned into at most $t$ pairwise twins (called graphs of {\em neighborhood diversity} at most $t$~\cite{Lampis2012}) has the rainbow property, where two vertices $v$ and $w$ in a graph $G$ are \emph{twins} if $v$ and $w$ have the same neighborhood in $G$ outside $\{v,w\}$.
In general, for a class $\mathcal{C}$ having the rainbow property, the class of all graphs obtained from a graph in $\mathcal{C}$ by replacing a vertex with a set of pairwise adjacent twins has the rainbow property. Thus, the class of graphs of neighborhood diversity at most $t$ has the rainbow property, because those graphs can be obtained from $K_{t+1}$-induced-subgraph-free graphs by replacing a vertex with a set of pairwise adjacent twins.

Another typical example of dense graph classes is the class of graphs of bounded \emph{clique-width}~\cite{CO2000} (equivalently, \emph{rank-width}~\cite{Oum2004}).
Briefly speaking, a graph has bounded clique-width if it can be decomposed into a tree-like structure where each vertex partition $(A, B)$ represented by an edge of the tree-structure has the property that $A$ has a bounded number of neighborhood types to $B$ and vice versa.
For example, graphs of bounded neighborhood diversity have small clique-width.
On the other hand, since complete multipartite graphs have clique-width $2$, the class of graphs of bounded clique-width (even restricted to the class of graphs with bounded linear clique-width, shrub-depth~\cite{GHNOO2017}, or modular-width~\cite{GajarskyLO2013}) does not have the rainbow property.

We are interested in finding new dense graph classes having the rainbow property.
For a graph $G$ and a positive integer $r$, the \emph{$r$-power} $G^r$  of $G$ is the graph obtained from $G$ by adding an edge between two vertices of distance at most $r$ in $G$.
Our main result is the following:
\begin{theorem}\label{thm:main2}
Let $\mathcal{C}$ be a bounded expansion class and $r$ be a positive integer. The class $\mathcal{D}=\{G^r : G\in \mathcal{C} \}$ has the rainbow property.
\end{theorem}
See Section~\ref{sec:prelim} for the definition of bounded expansion class. Bounded expansion classes include classes of $H$-minor free graphs, classes of graphs of bounded degree, classes of graphs of bounded stack number, and classes of graphs of bounded queue number.
We refer to \cite{NO2012c} for more examples of bounded expansion classes. The class of $d$-powers of graphs in some sparse graph class has been discussed recently, see~\cite{KwonPS2020, Gajarsky2020, Nesetril2020, Brianski2021}.

A \emph{map graph} is a graph that can be obtained from a plane graph by making a vertex for each face, and adding an edge between two vertices, if the corresponding faces share a vertex. 
Note that map graphs may have large cliques. 
Chen, Grigni, and Papadimitriou~\cite{ChenGP02} observed that every map graph is an induced subgraph of the $2$-power of a planar graph.
So, Theorem~\ref{thm:main2} implies that the class of map graphs has the rainbow property.
Contrary to Theorem~\ref{thm:main2}, the class of all possible graph powers of graphs in a bounded expansion class may not have the rainbow property. Aharoni et al.~\cite{ABKK2019} proved that the class of all possible powers of cycles does not have the rainbow property.

To prove Theorem~\ref{thm:main2}, we follow a similar approach used by Kwon, Pilipczuk, and Siebertz~\cite{KwonPS2020} to show that such class $\mathcal{D}$ admits low rank-width colorings.
In their proof, they used the fact that the class of $r$-powers of graphs of bounded tree-depth has bounded rank-width.
However, this does not help to show Theorem~\ref{thm:main2}, because complete multipartite graphs have rank-width at most $1$ while they do not have the rainbow property.
Instead, we show that the class of $r$-powers of graphs of bounded tree-depth has the rainbow property. 

We remark that Theorem~\ref{thm:main2} cannot be extended to $r$-powers of graphs of bounded degeneracy, because every graph can be obtained as an induced subgraph of the $2$-power of its $1$-subdivision, where this $1$-subdivision has degeneracy at most $2$. We leave as an open problem the question whether or not the class of $r$-powers of graphs in a nowhere dense class has the rainbow property.

The paper is organized as follows.
In Section~\ref{sec:prelim}, we introduce the necessary notions and notation used in the paper.
In Section~\ref{sec:treedepth}, we show that the class of $r$-powers of graphs of bounded tree-depth has the rainbow property, and based on this result, in Section~\ref{sec:expansion}, we prove that the class of $r$-powers of graphs in a bounded expansion class has the rainbow property.

\section{Preliminaries}\label{sec:prelim}
All graphs in this paper are finite, undirected graphs without loops or parallel edges.
Let $G$ be a graph. We write $V(G)$ for the vertex set of $G$ and $E(G)$ for its edge set.  
For $S \subset V(G)$, we denote by $G-S$ the graph obtained from $G$ by removing vertices in $S$ and all incident edges, 
and denote by $G[S]$ the subgraph of $G$ induced by $S$.
A graph $H$ is an \emph{induced subgraph} of $G$ if $H=G[S]$ for some $S\subseteq V(G)$.

The \emph{length} of a path is the number of edges in the path.
The {\em{distance}} between vertices $u$ and $v$ in $G$, denoted $\dist_G(u,v)$, is the length of a shortest path between~$u$ and $v$ in $G$, or $\infty$ if no such path exists.
The {\em radius} of $G$ is the minimum positive integer $d$ such that there is a vertex $u$ where $\dist_G(u,v) \leq d$ for every $v \in V(G)$.
The \emph{$r$-power of $G$} is the graph~$G^r$ with vertex set $V(G)$, where there is an edge between two vertices $u$ and $v$ in $G^r$ if and only if $\dist_G(u,v) \leq r$.
For a class $\mathcal{C}$ of graphs, we define $\cC^r:=\{G^r: G\in \cC\}$.

A set $I$ of vertices in $G$ is an \emph{independent set} if no two vertices in $I$ are adjacent.
A set $K$ of vertices in $G$ is a \emph{clique} if every two distinct vertices in $K$ are adjacent in $G$.
A set $M$ of edges in $G$ is a \emph{matching} if no two edges in $M$ share a vertex.
A set $S$ of vertices in $G$ is a \emph{vertex cover} if $G-S$ has no edges.

We define the tree-depth of a graph as in \cite{NO2012}.
A \emph{rooted forest} is a forest in which every connected component has  a specified node called a \emph{root}. 
A vertex $v$ is an \emph{ancestor} of a vertex $u$ in $T$ if the path from $u$ to the root in $T$ contains $v$. If $v$ is an ancestor of a vertex $u$ in $T$ and $uv\in E(T)$, then we say that $u$ is a \emph{child} of $v$.
The \emph{closure} of a rooted forest $T$, denoted by $\clos(T)$, is the graph obtained from $T$ by adding an edge between every vertex and all its ancestors.
The \emph{height} of a rooted forest is the number of vertices
in  a longest path from a root to a leaf.
The \emph{tree-depth} of a graph  $G$, denoted by $\td(G)$, is
the minimum height of a rooted forest whose closure contains $G$ 
as a subgraph. 

A \emph{bounded expansion} class was defined in terms of \emph{$t$-shallow minors}~\cite{NO2008a}.
A $t$-shallow minor of a graph $G$ is defined as a graph formed from $G$ by contracting a collection of vertex-disjoint subgraphs of radius $t$, and then deleting the remaining vertices of $G$. 
A class $\mathcal{C}$ of graphs has \emph{bounded expansion} if there exists a function $f\colon \mathbb{N}\to \mathbb{N}$ such that, in every $t$-shallow minor of a graph in $\mathcal{C}$, the ratio of edges to vertices is at most $f(t)$.
It turns out that low tree-depth colorings and weak colorings give alternative characterizations. 
As we use the two characterizations in the proof, we introduce them here. 

A \emph{vertex coloring} of a graph $G$ with a color set $S$ is a mapping $c\colon V(G) \to S$.  
For each $v\in V(G)$, we call $c(v)$ the color of $v$. 
A class $\cC$ of graphs \emph{admits low tree-depth colorings} if there exists a function $g:\mathbb{N}\rightarrow\mathbb{N}$ such that for all $p\in \mathbb{N}$, every graph $G\in \cC$ has a vertex coloring with at most $g(p)$ colors such that the union of any $i\leq p$ color classes induces a subgraph of tree-depth at most $i-1$.

For a graph $G$, we denote by $\Pi(G)$ the set of all linear orders of $V(G)$.  For $u,v\in V(G)$ and an integer $r\geq 0$, we say
that~$u$ is \emph{weakly $r$-reachable} from~$v$ with respect to a
linear order $L$, if there is a path $P$ of length at most~$r$ 
between~$u$ and $v$ such that $u$ is the smallest among the 
vertices of $P$ with respect to~$L$.
We denote by $\WReach_r[G,L,v]$ the set of vertices that are weakly $r$-reachable from~$v$ with respect to~$L$.
The \emph{weak $r$-coloring number $\wcol_r(G)$} of~$G$ is defined as
\begin{eqnarray*}
  \wcol_r(G)& := & \min_{L\in\Pi(G)}\:\max_{v\in V(G)}\:
                   \bigl|\WReach_r[G,L,v]\bigr|.
\end{eqnarray*}

\begin{theorem}[Ne\v{s}et\v{r}il and Ossona de Mendez~\cite{NO2008a}]\label{thm:tdcoloring}
A class of graphs has bounded expansion if and only if it admits low tree-depth colorings.
\end{theorem}

\begin{theorem}[Zhu~\cite{Zhu2009}]\label{thm:wcol}
A class $\cC$ of graphs has bounded expansion if and only if there is a function $f:\mathbb{Z}_{\geq0} \to \mathbb{Z}_{\geq0}$ such that for all $G\in \cC$, we have $\wcol_r(G)\le f(r)$.
\end{theorem}

\section{$r$-powers of graphs of bounded tree-depth}\label{sec:treedepth}
In this section, we prove that the class of $r$-powers of graphs of bounded tree-depth has the rainbow property, which is a core in the proof of Theorem~\ref{thm:main2}.
Let $\mathcal{T}_{d}$ be the class of graphs of tree-depth at most $d$.

\begin{theorem}\label{thm:treedepth}
Let $d, r, n$ be positive integers.
Then $f_{\mathcal{T}^r_{d}}(n)<\infty$.
\end{theorem}

For a positive integer $d$, let $\cF_{d}$ be the class of rooted trees $F$ of height $d$, and $\cT^*_d$ be the set of all subgraphs of graphs in $\{ \clos(F) : F\in \mathcal{F}_{d} \}$.
Note that $\mathcal{T}_d\subseteq \mathcal{T}^*_{d+1}$.
To prove Theorem~\ref{thm:treedepth}, we will prove $f_{\cT^*_{d}}(n) < \infty$ for all positive integers $d$ and $n$.

Let $F\in \cF_d$ and let $\cI$ be a family of vertex sets of $F$.
We define $t_{F}(\cI)\in \{1, 2, \ldots, d\}$ as the minimum integer $t$ satisfying the following: there is a node $v$ of distance $t-1$ from the root such that there are at least two subtrees, rooted at children of $v$ in $F$, meeting some vertex sets in $\cI$.
If there is no such $t$, then we define $t_{F}(\cI) = d$.

We use the following result about the class of graphs whose chromatic number is at most $k$.
\begin{theorem}[Aharoni, Briggs, J. Kim, and M. Kim~\cite{ABKK2019}]\label{thm:colorable}
Let $\cC$ be the class graphs whose chromatic number is at most $k$. Then $f_{\cC}(n)=k(n-1)+1$.
\end{theorem}

\begin{proposition}\label{prop:inductivenew}
For integers $d, n, r\ge 1$ and $0\le p\le d-1$, there exists an integer $M=M(d,n,p,r)$ satisfying the following.
Let $F$ be a rooted tree in $\mathcal{F}_{d}$ such that $G$ is a subgraph of $\clos(F)$, 
and $\cI$ be a set of $M$ independent sets of size $n$ in $G^r$ such that 
$p=d-t_{F}(\cI)$. Then $\cI$ contains a rainbow independent set of size $n$.
\end{proposition}
\begin{proof}
Let $q:=t_{F}(\cI)$.
We set
$M(d, 1, p, r):=1$ for all integers $d, r\ge 1$ and $0\le p\le d-1$, and 
$M(d, n, 0, r):=d(n-1)+1$ for all integers $d, r\ge 1$ and $n\ge 2$.
For integers $d,r,p\ge 1$ and $n\ge 2$, we set 
\begin{itemize}
\item $M_2:=\max (M(d, n-1,p,r)+n, M(d, n, p-1, r) )$,
\item $M_1:=M_2\left( (n+1)^{(n-1)(r+2)^{2q}} -1 \right)+(n-1)(r+2)^q+1$,
\item $M=M(d, n, p, r):=2^{(r+2)^q}(M_1-1)+1$.
\end{itemize}
We prove the proposition by induction on $n+p$. 
The case $n=1$ is obvious.
If $p=0$, then $t_{F}(\cI )=d$ and all the independent sets lie in a path $P$ from a leaf to the root in $F$. Then $G^r[V(P)]$ is a graph with chromatic number at most $d$, and the result follows from Theorem~\ref{thm:colorable}.
We assume $n\ge 2$ and $p\ge 1$.

Let $z$ be the root node of $F$, and $u$ be the node of $F$ such that $\dist_F(u, z)=t_{F}(\cI)-1$ and there are at least two subtrees rooted at children of $u$ in $F$, meeting some sets in $\cI$. 
 Let $v_1v_2 \cdots v_q$ be the path in $F$ from $z$ to $u$ where $z=v_1$ and $u=v_q$.
For convenience, we assume that $G$ contains all vertices of $\{v_1, \ldots, v_q\}$ by adding isolated vertices if necessary.
Let $\cH$ be the set of subtrees of $F$ rooted at children of $v_q$.
By the definition of $t_{F}(\cI)$, all the independent sets of $\cI$ are contained in $\{v_1, \ldots, v_q\}\cup \bigcup_{X\in \cH} V(X)$.

We classify vertices of $G$ by the distance to each vertex in $\{v_1, v_2, \ldots, v_q\}$.
For each vertex $v$ of $G$, let $d(v)$ be the vector $(d_1, d_2, \ldots, d_q)$ such that 
$d_i=\dist_G(v, v_i)$ if $\dist_G(v, v_i)\le r$ and $d_i=r+1$ otherwise. Note that $\dist_G(v, v_i)=\infty$ if $v$ and $v_i$ are not contained in the same connected component of $G$.
For each $\cA\in \{0, 1, \ldots, r+1\}^q$, 
let $U_{\cA}:=\{v\in V(G) : d(v)=\cA\}$. This is a partition of the vertex set of $G$ into $(r+2)^q$ sets.

 We use the following property of the sets $U_{\cA}$.
 \begin{claim}\label{claim:connection}
 Let $v, w$ be two vertices of $G$ such that 
\begin{itemize}
\item $v\in U_{(a_1, a_2, \ldots, a_q)}$ and $w\in U_{(b_1, b_2, \ldots, b_q)}$ for some vectors $(a_1, \ldots, a_q)$ and $(b_1, \ldots, b_q)$ in $\{0, 1, \ldots, r+1\}^q$.
\end{itemize}
If $a_i+b_i\le r$ for some $i\in \{1, 2, \ldots, q\}$, then $v$ is adjacent to $w$ in $G^r$.
On the other hand, if 
$v$ is contained in some subtree of $\cH$, say $Q$, and 
 $w$ is contained in $V(G)\setminus V(Q)$, and $a_i+b_i>r$ for all $i\in \{1, 2, \ldots, q\}$, 
 then $v$ is not adjacent to $w$ in $G^r$.
 \end{claim}
 \begin{clproof}
Suppose that $a_i+b_i\le r$ for some $i$. 
Then there is a path of length at most $r$ in $G$ from $v$ to $w$ through $v_i$.
Thus it is obvious that $v$ is adjacent to $w$ in $G^r$.

Now, suppose that 
$v$ is contained in some subtree of $\cH$, say $Q$, and 
 $w$ is contained in $V(G)\setminus V(Q)$, and
$v$ is adjacent to $w$ in $G^r$.
By definition, there is a path $P$ of length at most $r$ from $v$ to $w$ in $G$.
Because in $G-\{v_1, v_2, \ldots, v_q\}$, there is a no path from $V(Q)$ to $V(G)\setminus V(Q)$, 
the path $P$ meets $\{v_1, v_2, \ldots, v_q\}$.
Then 
we have that $a_i+b_i\le r$ for some $i\in \{1, 2, \ldots, q\}$.
 \end{clproof}

As $\abs{\cI}=M=2^{(r+2)^q}(M_1-1)+1$, by the pigeonhole principle, we can choose a subset $\cI_1$ of size at least $M_1$ such that 
\begin{itemize}
\item[($\ast$)] for every $I_1, I_2\in \cI_1$ and every vector $\mathcal{A}\in \{0, 1, \ldots, r+1\}^q$, 
$I_1\cap U_{\cA}=\emptyset$ if and only if $I_2\cap U_{\cA}=\emptyset$.
\end{itemize}

Now, for each vector $\mathcal{A}\in \{0, 1, \ldots, r+1\}^q$, 
we construct an auxiliary bipartite graph $B_{\mathcal{A}}$ on the bipartition $(\cH, \cI_1 )$ such that for $H\in \cH$ and $I\in \cI_1$, $H$ is adjacent to $I$ in $B_{\mathcal{A}}$ if and only if 
$I\cap U_{\cA}$ contains a vertex of $H$.
We divide cases depending on whether $B_{\mathcal{A}}$ contains a matching of size $n$ for some vector $\cA$ or not.

\smallskip

\noindent{\bf (Case 1.)} There exists a matching of size $n$ in $B_{\cA}$ for some vector $\cA=(a_1, a_2, \ldots, a_q)$, say $\{F_1I_1, F_2I_2, \ldots, F_nI_n\}$, where $F_1, \ldots, F_n\in \cH$ and $I_1, \ldots, I_n\in \cI_1$.

We choose a vertex $x_j$ in $I_j\cap V(F_j)\cap U_{\cA}$ for each $j$.
If $2a_i>r$ for all $i\in \{1, 2, \ldots, q\}$, then by Claim~\ref{claim:connection}, $x_i$ and $x_j$ are not adjacent for any $i \neq j$ since they are contained in distinct $V(F_j)$.
Thus $\{x_1, x_2, \ldots, x_n\}$ is a rainbow independent set, so we are done. 
Therefore, we may assume that 
$2a_i\le r$ for some $i\in \{1, 2, \ldots, q\}$. 
It implies that $U_{\cA}$ is a clique in $G^r$, because any two vertices in $U_{\cA}$ can be linked by a path of length at most $r$ through $v_i$.
Moreover, it implies that 
every independent set in $\cI_1$ contains exactly one vertex of $U_{\cA}$.

We define
\[\cI_1^*:=\{I\setminus U_{\cA}: I\in \cI_1\setminus \{I_1, I_2, \ldots, I_n\}\}. \]
Then $\abs{\cI_1^*}\ge M_1-n\ge M_2-n\ge M(d, n-1,p,r)$ and each independent set of $\cI_1^*$ has size exactly $n-1$.
As $\abs{\cI_1^*} \ge M(d, n-1,p,r)$,
$\cI_1^*$ contains a rainbow independent set $J$ of size $n-1$. 
As $J$ has size $n-1$, there exists $i \in \{1,2,\ldots,n\}$ such that $J$ does not meet $F_i$.
Without loss of generality, we assume that $J$ does not contain a vertex of $F_1$.

We claim that $x_1$ has no neighbor in $J$ in $G^r$, which implies that $\{x_1\}\cup J$ is a rainbow independent set in $G^r$.
Take any vertex $w$ in $J$.
Assume that $w$ is contained in $U_{\cB}$ for some vector $\cB=(b_1, b_2, \ldots, b_q)\in \{0, 1, \ldots, r+1\}^q$.
As $w$ is contained in $U_{\cB}$, by ($\ast$), $I_1$ also contains a vertex in $U_{\cB}$.
Since $I_1$ is independent, by Claim~\ref{claim:connection}, $a_i+b_i>r$ for all $i\in \{1, 2, \ldots, q\}$.
On the other hand, since $w$ is not in $F_1$, by Claim~\ref{claim:connection}, $x_1$ is not adjacent to $w$ in $G^r$.
So, we conclude that $\{x_1\}\cup J$ is a rainbow independent set, as required.

\smallskip

\noindent {\bf (Case 2.)} $B_{\cA}$ contains no matching of size $n$ for all vectors $\cA\in \{0, 1,  \ldots, r+1\}^q$.

By K\H{o}nig's Theorem, each $B_{\cA}$ contains a vertex cover $S_{\cA}$ of size at most $n-1$.
It means that all the independent sets that are contained in $\cI_1\setminus S_{\cA}$
lie in the union of $\{v_1, v_2, \ldots, v_q\}$ and the vertex sets of the subtrees in $S_{\cA}\cap \cH$. 
We define 
 \[ S:= \bigcup_{\cA \in \{0, 1, \ldots, r+1\}^q} S_{\cA}. \]
Collecting the information of each $S_{\cA}$, we can observe that 
\begin{enumerate}[(i)]
	\item $\abs{\cI_1\setminus S}\ge M_1-(n-1)(r+2)^q$ and $\abs{S\cap \cH} \le (n-1)(r+2)^q$, 
	\item every independent set in $\cI_1\setminus S$ lies in the union of $\{v_1, v_2, \ldots, v_q\}$ and the vertex sets of the subtrees in $S\cap \cH$.
\end{enumerate}

Now, we want to take a subset of $\cI_1$ so that the selected independent sets have the same number of vertices in each set $V(H)\cap U_{\cA}$ for $H\in S\cap \cH$ and $\cA\in \{0, 1, \ldots, r+1\}^q$.
Note that 
\[M_1-(n-1)(r+2)^q\ge M_2 \left( (n+1)^{(n-1)(r+2)^{2q}} -1 \right)+1.\]
So, by properties (i) and (ii), there exists a subset $\cI_2\subseteq \cI_1$ of size at least $M_2$
such that 
\begin{itemize}
\item[($\ast\ast$)] all independent sets in $\cI_2$ have the same number of intersections on $V(H)\cap U_{\cA}$ for each subtree $H$ in $S\cap \cH$ and each vector $\cA$ in $\{0, 1, \ldots, r+1\}^q$.
\end{itemize}

We assume that all the independent sets in $\cI_2$ lie in the union of $\{v_1, \ldots, v_q\}$ and exactly one subtree of $S\cap \cH$. It means that $t_{F}(\cI_2) > t_{F}(\cI)$. 
As $\abs{\cI_2}\ge M(d, n, p-1, r)\ge M(d, n, j, r)$ for all $0\le j\le p-1$, $\cI_2$ contains a rainbow independent set of size $n$. From now on, we assume that every independent set in $\cI_2$ meets at least two subtrees in $S\cap \cH$. 

Let $Q$ be a subtree in $S\cap \cH$ that meets an independent set in $\cI_2$, and $y$ be the number of intersections of each independent set with $Q$.
Let $\cI_2^*:=\{I\cap V(Q): I\in \cI_2\}$.
Observe that 
\begin{align}\label{ineq:M}M(d, n-1,p,r) \geq M(d, i,j,r)\end{align} 
for all $i \in \{1,2,\ldots,n-1\}$ and $j \in \{0,1,\ldots,p\}$.
Since \[\abs{\cI_2^*}=\abs{\cI_2} > M(d, n-1,p,r),\] by the induction hypothesis and \eqref{ineq:M}, $\cI_2^*$ contains a rainbow independent set $R_1 = \{z_{i_1}, z_{i_2},\ldots,z_{i_y}\}$ of size $y < n$ such that $z_{i_j} \in I_{i_j} \in \cI_2^*$.
Let \[\cI_2^+:=\{ I\setminus V(Q) : I\in \cI_2 \setminus \{I_{i_1},I_{i_2}\ldots,I_{i_y}\}\}.\]
Note that each set in $\cI_2^+$ has size $n-y$.
Since \[\abs{\cI_2^+} = \abs{\cI_2} - y \geq M(d,n-1,p,r) +n - y > M(d,n-1,p,r),\]
by the induction hypothesis and \eqref{ineq:M}, $\cI_2^+$ also contains a rainbow independent set $R_2$ of size $n-y$.
In particular, $R_1 \cup R_2$ is a rainbow set for $\cI_2$.

We claim that $R_1\cup R_2$ is a rainbow independent set of size $n$.
Let $w_1\in R_1$ and $w_2\in R_2$.
Assume $w_1\in U_{\cB_1}, w_2\in U_{\cB_2}$ for some vectors $\cB_1=(b^1_1, b^1_2, \ldots, b^1_q)\in \{0, 1, \ldots, r+1\}^q$ and $\cB_2=(b^2_1, b^2_2, \ldots, b^2_q)\in \{0, 1, \ldots, r+1\}^q$.
Let $I$ be an independent set of $\cI_2$ containing $w_1$.
Since $w_2$ is not in $Q$, it can be either in another tree $Q'$ from $\mathcal{H}$, or it can be $w_2=v_j$ for some $j\in \{1, \ldots, q\}$. In the latter case, we have $b^2_j=0$, and so by the property $(\ast)$ of $\cI_1$, it holds that $w_2$ is also in $I'$.
This means that $w_1$ is not adjacent to $w_2$, as desired.
Thus, we may assume that $w_2$ is contained in some subtree $Q'$ from $\cH$ such that $Q' \neq Q$.

As $w_2$ is contained in $V(Q')\cap U_{\cB_2}$, by the property of $\cI_2$, $I$ also contains a vertex in $V(Q')\cap U_{\cB_2}$.
Since $I$ is independent, by Claim~\ref{claim:connection}, $b^1_i+b^2_i>r$ for all $i\in \{1, 2, \ldots, q\}$.
On the other hand, since $w_1$ and $w_2$ are contained in distinct subtrees of $\cH$, by Claim~\ref{claim:connection}
$w_1$ is not adjacent to $w_2$ in $G^r$, as required.
We conclude that $R_1\cup R_2$ is a rainbow independent set of size $n$.
\end{proof}

\begin{proof}[Proof of Theorem~\ref{thm:treedepth}]
Let $M$ be the function in Proposition~\ref{prop:inductivenew}.
We claim that for every positive integer $d$, we have 
$f_{\mathcal{T}^r_{d}}(n)\le M(d+1, n, d+1, r)$.
Let $G$ be a tree-depth at most $d$ and $T$ be the rooted forest whose closure contains $G$ as a subgraph.
Let $F$ be a rooted tree obtained from $T$ by adding an isolated vertex, which is a new root, and adding edges between the new vertex and original roots in the components of $T$. Then clearly, $G$ is a subgraph of $\clos(F)$ and $F\in \mathcal{F}_{d+1}$.
Let $\cI$ be a set of $M(d+1, n, d+1, r)$ independent sets of size $n$ in $G^r$.
Then by Proposition~\ref{prop:inductivenew}, $\cI$ contains a rainbow independent set of size $n$.
\end{proof}

\section{$r$-powers of graphs of bounded expansion}\label{sec:expansion}

In this section, we prove that the class of $r$-powers of graphs in a bounded expansion class has the rainbow property.
\begin{thm:main2}
Let $\mathcal{C}$ be a bounded expansion class and $r$ be a positive integer. The set $\mathcal{D}=\{G^r : G\in \mathcal{C} \}$ has the rainbow property.

\end{thm:main2}

  We can start with low tree-depth colorings of a graph $G$ in $\mathcal{C}$ whose existence is guaranteed by Theorem~\ref{thm:tdcoloring}.
  Note that the size $n$ of an independent set we are dealing is given, 
  we will take a low tree-depth coloring where the union $S$ of $n$ color classes have small tree-depth.
  But when we take the $r$-power of $G$, 
  $G^r[S]$ is not necessarily the same as $G[S]^r$, because the two vertices in $S$ can be adjacent in $G^r$ because of a path going through outside $S$.
  In fact, the similar problem happens in the result of Kwon, Pilipczuk, and Siebertz~\cite{KwonPS2020}, 
  and they resolve this problem by introducing a notion of $r$-shortest path closures.

	For a graph $G$, $X\subseteq V(G)$, and a positive integer $r$, a superset $X'$ of $X$ is called an \emph{$r$-shortest path closure} of $X$ 
	if for each $u,v\in X$ with $\dist_G(u,v)=\ell\le r$, $G[X']$ contains a path of length $\ell$ between $u$ and $v$.
	For a graph $G$, a coloring $c$ of $G$, and integers $r\ge 2$, $d\ge 1$, 
	a coloring $c'$ is a $(d,r)$-\emph{excellent refinement} of $c$ if for every vertex set $X\subseteq V(G)$,
	there exists an $r$-shortest path closure $X'$ of $X$ such that if $X$ receives at most $p$ colors in $c'$, then $X'$ receives at most $d\cdot p$ colors in $c$.

	\begin{lemma}[Kwon, Pilipczuk, and Siebertz~\cite{KwonPS2020}]\label{lem:excellent}
	Let $G$ be a graph, let $k\ge 1$, $r\ge 2$ be integers, and let $d_r:=\prod_{2\le \ell\le r} 2\wcol_\ell (G)$.
	Then every coloring $c$ of $G$ using at most $k$ colors has a $(d_r, r)$-excellent refinement coloring using at most $k^{d_r}$ colors.
	\end{lemma}
	
	\begin{proof}[Proof of Theorem~\ref{thm:main2}]
	Let $n$ be a positive integer. We have to show that there exists $N$ such that 
	for every graph $G$ in $\mathcal{C}$, $f_{G^r}(n)\le N$. 
		
	By Theorem~\ref{thm:wcol}, for each $\ell$, $\wcol_{\ell}(G)$ is bounded by a constant, say $W_{\ell}$, only depending on $\cC$.
	Also, by Theorem~\ref{thm:tdcoloring}, 
	there exists a function $g:\mathbb{N}\rightarrow\mathbb{N}$
  such that for all $p\in \mathbb{N}$, every graph $G\in \cC$ can be vertex
  colored with at most $g(p)$ colors such that the union of any
  $i\leq p$ color classes induces a subgraph of tree-depth at most
  $i-1$.
  	Let  $d_r:=\prod_{2\le \ell\le r} 2W_{\ell}$, and
	$L:= f_{\mathcal{T}^r_{d_r \cdot n}}(n)$ which is finite by Theorem~\ref{thm:treedepth}.
	Finally, we set $N:=g(d_r\cdot n)^{d_rn}\cdot L$.

	Let $G\in \cC$, and 
	$\cI$ be a set of independent sets of $G^r$ of size $n$ such that $\abs{\mathcal{I}}\ge N$.
	
	Let $c$ be a $(d_r\cdot n)$-tree-depth coloring with $g(d_r \cdot n)$ colors.
	We take a $(d_r, r)$-excellent refinement $c'$ of $c$ with at most $g(d_r\cdot n)^{d_r}$ colors by Lemma~\ref{lem:excellent}.
	
	Since $\abs{\mathcal{I}}\ge g(d_r\cdot n)^{d_rn}\cdot L$, 
	there exist $n$ color classes $X_1, X_2, \ldots, X_n$ of $c'$ and a subset $\mathcal{I}'\subseteq \mathcal{I}$ such that 
	$\abs{\mathcal{I'}}\ge L$ and every independent set of $\mathcal{I}'$ is contained in $X_1\cup X_2\cup \cdots \cup X_n$. 
	As $c'$ is a $(d_r, r)$-excellent refinement of $c$, 
	there exists an $r$-shortest path closure $X'$ of $X_1\cup X_2\cup\cdots \cup X_n$ that uses at most $d_r\cdot n$ colors of $c$.
	Thus, the graph induced by $X_1\cup X_2\cup \cdots \cup X_n$ in $G^r$ is the same as the graph 
	induced by the same set in $G[X']^r$. 
	Since $G[X']$ has tree-depth at most $d_r\cdot n$ and $\abs{\mathcal{I'}}\ge L=f_{\mathcal{T}^r_{d_r \cdot n}}(n)$,
	by Theorem~\ref{thm:treedepth}, $\mathcal{I}'$ contains a rainbow independent set.
	\end{proof}
	
	We discuss some applications of Theorem~\ref{thm:main2}.
	We denote by $\overline{nK_2}$ the complement of the $n$ disjoint union of $K_2$, which is the complete multipartite graph where it has $n$ parts and each part has size $2$.
	It is not difficult to see that $\{\overline{nK_2}: n\in \mathbb{N}\}$ has infinite $f_{\cC}(2)$.
	Thus, we can deduce the following.
\begin{corollary}
Let $\mathcal{C}$ be a bounded expansion class and $r$ be a positive integer. Then there exists a positive integer $n$ such that every graph in $\{G^r : G\in \mathcal{C} \}$ does not contain an induced subgraph isomorphic to $\overline{nK_2}$.
\end{corollary}

	 As we discussed in the introduction, 
	map graphs can be obtained as  induced subgraphs of $2$-powers of planar graphs. Thus, by Theorem~\ref{thm:main2}, 
	the class of map graphs has the rainbow property.
	\begin{corollary}
	The class of map graphs has the rainbow property.
	\end{corollary}
	
	We may also observe that induced matchings in a bounded expansion class have the rainbow property.
	A matching $M$ is \emph{induced} if for distinct edges $a_1b_1, a_2b_2\in M$, 
	there are no edges between $\{a_1, b_1\}$ and $\{a_2, b_2\}$.
	
	\begin{corollary}
	Every bounded expansion class satisfies the rainbow property for induced matchings.
	\end{corollary}
	\begin{proof}
	Let $\cC$ be a bounded expansion class, $n$ be an integer.
	We consider the class $\mathcal{D}$  which consists of $1$-subdivisions of all graphs in $\cC$.
	It is well known that $\mathcal{D}$ also has bounded expansion.
	Take $M:=f_{\mathcal{D}^5}(n)$.
	
	Let $G$ be a graph in $\cC$, and 
	$\cI$ be a family of $M$ induced matchings of size $n$ in $G$.
	Let $H$ be the $1$-subdivision of $G$. 
	For each induced matching $I \in \cI$ and for each $e$ of $G$ in $I$, we consider the corresponding subdivision vertex in $H$ and obtain an independent set in $H$. 
	Let $\cI^*$ be the family of resulting independent sets.
	
	We observe that for $I\in \cI^*$ and two distinct vertices $v,w\in I$, $\dist_H(v,w)\ge 6$ as they came from an induced matching of $G$.
	It means that $\cI^*$ is a family of independent sets in the $5$-power $H^5$ of $H$ as well.
	As $\abs{ \cI^* }\ge M=f_{\mathcal{D}^5}(n)$, $\cI^*$ contains a rainbow independent set $X$ of size $n$ in $H^5$.
	Now, we obtain an edge set $Y$ in $G$ by taking the original edge from the subdivided vertices contained in $X$.
	Since $X$ is an independent set of $H^5$, 
	$Y$ is again an induced matching of $G$.
	\end{proof}

\paragraph*{Acknowledgement.}
The authors would like to thank anonymous reviewers for their helpful suggestions.      
The authors also thank Joseph Briggs for his comment on the argument about graphs of bounded neighborhood diversity.


\begin{thebibliography}{10}

\bibitem{RonE2009}
R.~Aharoni and E.~Berger.
\newblock Rainbow matchings in {$r$}-partite {$r$}-graphs.
\newblock {\em Electron. J. Combin.}, 16(1):Research Paper 119, 9, 2009.

\bibitem{RonEMDP2019}
R.~Aharoni, E.~Berger, M.~Chudnovsky, D.~Howard, and P.~Seymour.
\newblock Large rainbow matchings in general graphs.
\newblock {\em European J. Combin.}, 79:222--227, 2019.

\bibitem{ABKK2019}
R.~Aharoni, J.~Briggs, J.~Kim, and M.~Kim.
\newblock Rainbow independent sets in certain classes of graphs.
\newblock arXiv:1909.13143, 2019.

\bibitem{AharoniBKK2021}
R.~Aharoni, J.~Briggs, J.~Kim, and M.~Kim.
\newblock Badges and rainbow matchings.
\newblock {\em Discrete Math.}, 344(6):112363, 2021.

\bibitem{Brianski2021}
M.~Bria\'{n}ski, P.~Micek, M.~Pilipczuk, and M.~T. Seweryn.
\newblock Erd\"{o}s--{H}ajnal properties for powers of sparse graphs.
\newblock {\em SIAM J. Discrete Math.}, 35(1):447--464, 2021.

\bibitem{ChenGP02}
Z.~Chen, M.~Grigni, and C.~H. Papadimitriou.
\newblock Map graphs.
\newblock {\em J. {ACM}}, 49(2):127--138, 2002.

\bibitem{CO2000}
B.~Courcelle and S.~Olariu.
\newblock Upper bounds to the clique width of graphs.
\newblock {\em Discrete Appl. Math.}, 101(1-3):77--114, 2000.

\bibitem{Drisko1998}
A.~A. Drisko.
\newblock Transversals in row-{L}atin rectangles.
\newblock {\em J. Combin. Theory Ser. A}, 84(2):181--195, 1998.

\bibitem{Gajarsky2020}
J.~Gajarsk\'{y}, S.~Kreutzer, J.~Ne\v{s}et\v{r}il, P.~Ossona~de Mendez,
  M.~Pilipczuk, S.~Siebertz, and S.~Toru\'{n}czyk.
\newblock First-order interpretations of bounded expansion classes.
\newblock {\em ACM Trans. Comput. Log.}, 21(4):Art. 29, 41, 2020.

\bibitem{GajarskyLO2013}
J.~Gajarsk\'{y}, M.~Lampis, and S.~Ordyniak.
\newblock Parameterized algorithms for modular-width.
\newblock In {\em Parameterized and exact computation}, volume 8246 of {\em
  Lecture Notes in Comput. Sci.}, pages 163--176. Springer, Cham, 2013.

\bibitem{GHNOO2017}
R.~Ganian, P.~Hlin{\v{e}}n{{\'y}}, J.~Ne{\v{s}}et{\v{r}}il,
  J.~Obdr\v{z}\'{a}lek, and P.~Ossona~de Mendez.
\newblock Shrub-depth: Capturing height of dense graphs.
\newblock {\em Log. Methods Comput. Sci.}, 15(1):7:1--7:25, 2019.

\bibitem{KimL2019}
M.~Kim and A.~Lew.
\newblock Complexes of graphs with bounded independence number.
\newblock {\em Israel J. Mathematics}, accepted.
\newblock arXiv:1912.12605.

\bibitem{KwonPS2020}
O.~Kwon, M.~Pilipczuk, and S.~Siebertz.
\newblock On low rank-width colorings.
\newblock {\em European J. Combin.}, 83:103002, 17, 2020.

\bibitem{Lampis2012}
M.~Lampis.
\newblock Algorithmic meta-theorems for restrictions of treewidth.
\newblock {\em Algorithmica}, 64(1):19--37, 2012.

\bibitem{NO2008a}
J.~Ne{\v{s}}et{\v{r}}il and P.~Ossona~de Mendez.
\newblock Grad and classes with bounded expansion. {I}. {D}ecompositions.
\newblock {\em European J. Combin.}, 29(3):760--776, 2008.

\bibitem{NO2011}
J.~Ne{\v{s}}et{\v{r}}il and P.~Ossona~de Mendez.
\newblock On nowhere dense graphs.
\newblock {\em European J. Combin.}, 32(4):600--617, 2011.

\bibitem{NO2012}
J.~Ne{\v{s}}et{\v{r}}il and P.~Ossona~de Mendez.
\newblock {\em Sparsity}, volume~28 of {\em Algorithms and Combinatorics}.
\newblock Springer, Heidelberg, 2012.

\bibitem{Nesetril2020}
J.~Ne\v{s}et\v{r}il, P.~Ossona~de Mendez, M.~Pilipczuk, and X.~Zhu.
\newblock Clustering powers of sparse graphs.
\newblock {\em Electron. J. Combin.}, 27(4):P4.17, 2020.

\bibitem{NO2012c}
J.~Ne\v{s}et\v{r}il, P.~Ossona~de Mendez, and D.~R. Wood.
\newblock Characterisations and examples of graph classes with bounded
  expansion.
\newblock {\em European J. Combin.}, 33(3):350--373, 2012.

\bibitem{Oum2004}
S.~Oum.
\newblock Rank-width and vertex-minors.
\newblock {\em J. Combin. Theory Ser. B}, 95(1):79--100, 2005.

\bibitem{Zhu2009}
X.~Zhu.
\newblock Colouring graphs with bounded generalized colouring number.
\newblock {\em Discrete Math.}, 309(18):5562--5568, 2009.

\end{thebibliography}
\end{document}